\documentclass{article}
\usepackage{amsmath}
\usepackage{amssymb}
\usepackage{amsthm}
\usepackage{enumerate}
\usepackage{hyperref}

\newtheorem{theorem}{Theorem}[section]
\newtheorem{lemma}[theorem]{Lemma}
\newtheorem{proposition}[theorem]{Proposition}
\theoremstyle{remark}
\newtheorem*{remark}{Remark}

\numberwithin{equation}{section}

\allowdisplaybreaks

\newcommand{\NN}{\mathbb N}                             
\newcommand{\ZZ}{\mathbb Z}                             
\newcommand{\RR}{\mathbb R}                             
\newcommand{\CC}{\mathbb C}                             
\newcommand{\ra}{\rightarrow}                           
\newcommand{\abs}[1]{\left| #1 \right|}                 
\newcommand{\floor}[1]{\left\lfloor #1 \right\rfloor}   
\newcommand{\p}[1]{\left( #1 \right)}                   
\DeclareMathOperator{\e}{\mathrm{e}}                    
\title{Normality of the Thue--Morse sequence along Piatetski-Shapiro
sequences}
\author{Lukas~Spiegelhofer\medskip\\
Institute of Discrete Mathematics and Geometry\\
Vienna University of Technology, Vienna, Austria}
\begin{document}
\maketitle
\begin{abstract}
We prove that for $1<c<4/3$ the subsequence of the Thue--Morse sequence
$\mathbf t$ indexed by $\lfloor n^c\rfloor$ defines a normal sequence, that is,
each finite sequence $(\varepsilon_0,\ldots,\varepsilon_{T-1})\in \{0,1\}^T$
occurs as a contiguous subsequence of the sequence
$n\mapsto \mathbf t\left(\lfloor n^c\rfloor\right)$ with asymptotic frequency
$2^{-T}$.
\end{abstract}
\section{Acknowledgement}
This is a pre-copyedited, author-produced version of an article accepted for publication in The Quarterly Journal of Mathematics following peer review.
The version of record (Q. J. Math (2015) 66 (4): 1127-1138) is available online at:  https://doi.org/10.1093/qmath/hav029.
\section{Introduction}
The Thue--Morse sequence $\mathbf t$ has a rich history and has been studied
from various viewpoints, see for example the article~\cite{AS99} by Allouche
and Shallit or the article~\cite{M2001} by Mauduit. One approach to this
sequence is to view it as a $2$-automatic sequence (indeed it is one of the
simplest such sequences), which means that it is the output of a finite
automaton which is fed by the binary representations of $0,1,2,\ldots$. The
sequence $\mathbf t$ encodes the number of $1$'s in the binary representation
of a nonnegative integer $n$, which we write $s(n)$, reduced modulo $2$.
Alternatively, it can be described as the unique fixed point of the
substitution $0\mapsto 01$, $1\mapsto 10$ starting with the symbol $0$,
yielding \[\mathbf t=01101001100101101001011001101001\ldots\]

Since $\mathbf t$ is an automatic sequence --- we refer the reader to the
book~\cite{AS2003} by Allouche and Shallit for a comprehensive treatment of
automatic sequences --- its subword complexity is sublinear. In other words,
there exists a constant $C$ such that for each positive integer $k$ the
sequence $\mathbf t$ has at most $Ck$ different subwords (contiguous finite
subsequences) of length $k$.
In particular, since there are $2^k$ possible subwords of length $k$, the
sequence $\mathbf t$ is not normal. Moreover, $\mathbf t$ is an example of a
uniformly recurrent sequence, that is, for each subword $w$ appearing in
$\mathbf t$ there is a $k$ such that each subword of $\mathbf t$ of length $k$
contains $w$.

Despite these properties indicating the ``simplicity'' of the Thue--Morse
sequence, Drmota, Mauduit and Rivat~\cite{DMR2014} could recently prove that
the subsequence of $\mathbf t$ indexed by the squares is normal. That is, they
proved that each finite sequence $(\varepsilon_0,\ldots,\varepsilon_{T-1})$ in
$\{0,1\}$ occurs as a subword of the sequence $n\mapsto \mathbf t(n^2)$ with
asymptotic density $2^{-T}$. This result is remarkable not only for the reason
that we obtain a normal sequence by applying a simple operation to a
readily defined sequence, but also for the reason that this normal sequence
can be constructed \emph{bit-wise} via a simple algorithm --- which consists of
counting, modulo two, the $1$'s in the binary expansion of $n^2$.

In this work, we show that such a normality result can also be achieved for
certain subsequences $n\mapsto \mathbf t\p{\floor{f(n)}}$, where
$f:\NN\ra\RR^+$ grows more slowly than the sequence of squares. In particular,
we are interested in Piatetski-Shapiro sequences $n\mapsto \floor{n^c}$ with
$1<c<2$. The study of the behavior of digital functions (such as the
Thue--Morse sequence, for example) on $\floor{n^c}$ was initiated by the
article~\cite{MR95} by Mauduit and Rivat. They proved that subsequences
$n\mapsto \varphi\p{\floor{n^c}}$ of $q$-multiplicative functions $\varphi$
with values in $\{z\in\CC:\abs{z}=1\}$ behave as expected, given that
$1<c<4/3$, a bound that they improved to $1<c<7/5$ in~\cite{MR2005}. For a
definition of the term $q$-multiplicative function we refer to the cited
articles~\cite{MR95,MR2005}, we only note that the Thue--Morse sequence in the
form $n\mapsto (-1)^{s(n)}$ is such a function. The author~\cite{S2014}
replaced the bound $7/5=1.4$ by $1.42$ for the special case that
$\varphi=\mathbf t$.

Deshouillers, Drmota and Morgenbesser~\cite{DDM2012} studied the behavior of
automatic sequences $\mathbf u$ on $\floor{n^c}$, establishing the result that
if the density of each letter $a$ in the sequence $\mathbf u$ exists, then the
same is true for subsequence $n\mapsto \mathbf u\p{\floor{n^c}}$ and the
corresponding densities are identical, given that $c<7/5$. Moreover, they
considered consecutive terms in subsequences of the Thue--Morse sequence,
proving that for $1<c<10/9$ the density of each block
$(\varepsilon_0,\varepsilon_1)\in \{0,1\}^2$ of length two in the sequence
$n\mapsto \mathbf t\p{\floor{n^c}}$ equals $1/4$. An inspection of the method
used for obtaining this latter result shows that an analogous result can be
proved for arbitrary block lengths greater than two --- however, the interval
of admissible values for $c$ shrinks as the block length grows. In particular,
this method is not sufficient to prove the normality of Piatetski-Shapiro
subsequences of $\mathbf t$.

In the present article, we prove that for $1<c<4/3$ the subsequence of
$\mathbf t$ indexed by $\floor{n^c}$ is indeed a normal sequence. More
generally, our method allows us to conclude that the Thue--Morse sequence is
normal along $\floor{f(n)}$, where the second derivative of $f$ satisfies
some restrictions on the growth rate. Examples of such functions include
$\floor{n^c(\log n)^\eta}$ for $1<c<4/3$ and $\eta\in\RR$, or certain functions
growing more slowly than $n^{1+\varepsilon}$ for all $\varepsilon>0$.

Throughout this paper, we use the $1$-periodic functions $\e(x)=\exp(2\pi i x)$
and $\{x\}=x-\floor{x}$. Constants implied by the symbols $\ll$, $\asymp$ and
$O$ are absolute.
\section{The main result}
For the sake of better readability, we state the main theorem only for the
function $\lfloor n^c\rfloor$. The proof, however, deals with the more general
case announced in the introduction.
\begin{theorem}
Let $1<c<4/3$. The sequence
\[
n\mapsto \mathbf t\p{\floor{n^c}}
\]
is normal.
\end{theorem}
In order to prove this theorem, it is sufficient to establish the following
exponential sum estimate. (Compare to~\cite[Theorem 2]{DMR2014}).
\begin{proposition}\label{prp:1}
Let $T\geq 1$ be an integer, $\alpha_0=1$ and
$\alpha_1,\ldots,\alpha_{T-1}\in \{0,1\}$.
Assume that $1<c<4/3$.
Then, as $A\ra\infty$,
\begin{equation}
\sum_{A<n\leq 2A}
\e\p{
\frac 12
\sum_{\ell<T}
\alpha_\ell
s\p{\floor{(n+\ell)^c}}
}
=o(A).
\end{equation}
\end{proposition}
The plan of the remaining pages is as follows: we first state a series of
lemmas, of which the first one is the most important one, that we use in the
proof of Proposition~\ref{prp:1}. Afterwards, we give a short explanation of
the idea of proof which is intended to facilitate understanding. That paragraph
is then followed by the proof itself.
\section{Auxiliary results}
Let $\lambda\geq 0$ be an integer.
The \emph{truncated sum-of-digits function} $s_\lambda$ is defined by
\[
s_\lambda(n)=s\p{n\bmod 2^\lambda}
.
\]
For sequences of integers $\beta=(\beta_0,\ldots,\beta_{L-1})\in\{0,1\}^L$ and
$i=(i_0,\ldots,i_{L-1})$ we define discrete Fourier coefficients
\begin{equation}\label{eqn:fourier_coefficients}
G_\lambda^{i,\beta}(h,d)
=
\frac 1{2^\lambda}
\sum_{u<2^\lambda}
\e\p{
\frac 12
\sum_{\ell<L}
\beta_\ell s_\lambda(u+\ell d+i_\ell)
-hu2^{-\lambda}
}.
\end{equation}
One of the main ingredients in our proof is the following estimate for these
Fourier coefficients that is used as an essential tool in the proof of the
normality result for squares~\cite{DMR2014} and which is Proposition 1 in that article.
\begin{lemma}[Drmota, Mauduit, Rivat]\label{lem:DMR}
Let $L\geq 1$ be an integer. For all sequences
$\beta=(\beta_0,\ldots,\beta_{L-1})\in\{0,1\}^L$ satisfying $\beta_0=1$ and
$\beta_0+\cdots+\beta_{L-1}\in 2\ZZ$ and for all $i=(i_0,\ldots,i_{L-1})$
satisfying $i_0=0$ and $i_{\ell+1}-i_\ell\in\{0,1\}$, there exist $\eta>0$ and
$c_0$ such that for all integers $h$ and $\lambda,\lambda'\geq 0$ with
$\lambda/2\leq \lambda'\leq \lambda$ we have
\[
\frac 1{2^{\lambda'}}
\sum_{0\leq d<2^{\lambda'}}
\abs{G_\lambda^{i,\beta}(h,d)}^2
\leq
c_0
2^{-\eta\lambda}.
\]
\end{lemma}
In the proof of Proposition~\ref{prp:1} we will be concerned with more general
functions than just $x\mapsto x^c$. The following lemmas are dealing with such
functions $f$. The first lemma is a consequence of the Erd\H{o}s--Tur\'an
inequality combined with an exponential sum estimate due to van der Corput.
\begin{lemma}\label{lem:f_discrepancy}
Let $J$ be an interval in $\RR$ containing $N$ integers and let $f:J\ra\RR$ be
twice continuously differentiable in $J$. Let $\Lambda>0$ be such that
\[\Lambda\leq \abs{f''(x)}\leq 2\Lambda\]
for all $x\in J$.
Assume that $m,s$ and $M$ are nonnegative integers such that $m<M$. Then for
all integers $K,L\geq 1$ we have
\begin{multline*}
\abs{\{n\in J:\frac mM\leq \{f(n)\}<\frac{m+1}M,\floor{f(n)}\equiv s\bmod K\}}
\\
=
\frac N{KM}
+
O\p{
\frac N{KLM}
+
N\p{\Lambda LM}^{1/2}
+
\p{\Lambda LM}^{-1/2}
}
.
\end{multline*}
\end{lemma}
\begin{proof}
Assume without loss of generality that $0\leq s<K$.
We set
\[I=\left[\frac{Ms+m}{KM},\frac {Ms+m+1}{KM}\right),\]
which is an interval in $[0,1)$.
Then by the Erd\H{o}s--Tur\'an inequality we have
\begin{multline*}
\abs{
\abs{
\left\{n\in J:m/M\leq \{f(n)\}<(m+1)/M,
\floor{f(n)}\equiv s\bmod K
\right\}
}
-
\frac N{KM}
}
\\
=
\abs{
\abs{
\left\{n\in J:\{f(n)/K\}\in I\right\}
}
-
\frac N{KM}
}
\ll
\frac NH
+
\sum_{1\leq h\leq H}
\frac 1h
\abs{
\sum_{n\in J}
\e\p{hf(n)/K}
}
.
\end{multline*}
Applying Theorem 2.2 from~\cite{GK91}, which is the announced exponential sum
estimate due to van der Corput, yields
\[
\sum_{n\in J}
\e\p{hf(n)/K}
\ll
N\Lambda^{1/2}\p{h/K}^{1/2}+\Lambda^{-1/2}\p{K/h}^{1/2}
.
\]
The statement therefore follows from the choice $H=KLM$ and the estimate
$\sum_{1\leq h\leq H}h^\delta\ll H^{\delta+1}$ that holds for
$\delta\in\{-3/2,-1/2\}$.
\end{proof}
The second lemma provides us with some elementary estimates concerning the
derivatives of functions $f$ that we are interested in.
\begin{lemma}\label{lem_properties_f}
Let $x_0\geq 1$.
Assume that $f:[x_0,\infty)\ra\RR$ is two times continuously differentiable,
$f,f',f''>0$
and that for $x_0\leq x\leq y\leq 2x$ we have
$f''(x)/2\leq f''(y)\leq f''(x)$.
Then the following estimates hold.
\begin{alignat}{3}
  xf''(x) &\leq 2yf''(y)
  &&
    \text{ for }x_0\leq x\leq y,
    \label{eqn:almost_monotone}
\\
    xf''(x) &\leq 2f'(x)
    \quad
  &&
    \text{ for }x\geq 2x_0,
    \label{eqn:df_d2f_1}
\\
    f'(x)-f'(x_0)&\leq 2xf''(x)\log x
    \quad
  &&
    \text{ for }x\geq x_0,
    \label{eqn:df_d2f_2}
\\
    f'(y)&\leq 3 f'(x)
  &&
    \text{ for }2x_0\leq x\leq y\leq 2x.
    \label{eqn:df_quotient}
\end{alignat}
\end{lemma}
\begin{proof}
In order to prove~\eqref{eqn:almost_monotone}, we show the equivalent statement
that
\[
f''(x)\leq 2af''(ax)
\]
for $a\geq 1$ and $x\geq x_0$.
This is clear for $a=2^k$ by the inequality $f''(x)/2\leq f''(2x)$.
If $2^k\leq a<2^{k+1}$, we have
$f''(ax)
\geq f''(2^kx)/2
\geq 2^{-k}f''(x)/2
\geq f''(x)/(2a)$.
This proves~\eqref{eqn:almost_monotone}.
Inequality~\eqref{eqn:df_d2f_1} is proved via the Mean Value Theorem and the
monotonicity of $f''$: there exists some $\xi\in [x/2,x]$ such that
$f'(x)
\geq
f'(x)-f'(x/2)
=
(x/2)f''(\xi)
\geq
xf''(x)/2
$.
For the proof of~\eqref{eqn:df_d2f_2}, let $x\geq x_0$. For $x_0\leq t\leq x$
we have $tf''(t) \leq 2xf''(x)$ by~\eqref{eqn:almost_monotone}
and therefore
\[
    f'(x)=f'(x_0)+\int_{x_0}^xf''(t)\,\mathrm d t
  \leq
    f'(x_0)+2xf''(x)\int_{x_0}^x\frac 1t\,\mathrm d t
  \leq
    f'(x_0)+2xf''(x)\log x
.
\]
Finally, we prove~\eqref{eqn:df_quotient}: there exists $\xi\in [x,2x]$ such
that $f'(2x)-f'(x)=xf''(\xi)\leq xf''(x)$. Together with~\eqref{eqn:df_d2f_1}
we get $f'(y)\leq f'(2x)\leq 3f'(x)$.
\end{proof}
Assume that $f$ is a function such as in Lemma~\ref{lem_properties_f}.
Since $f''>0$, the first derivative $f'$ is increasing in $[x_0,\infty)$.
Moreover, setting $x=x_0$ in~\eqref{eqn:almost_monotone}, we get $f''(y)\gg 1/y$, therefore $f'(y)\gg \log y$, in particular $f'(y)\ra\infty$ as $y\ra\infty$.
This clearly implies $f(y)\ra\infty$.
We will use this observation in the proof of the following
elementary carry propagation lemma.
Statements of this type were used in the articles~\cite{MR2009,MR2010} by
Mauduit and Rivat on the sum-of-digits function of primes and squares.
\begin{lemma}\label{lem:f_lambda}
Let $x_0\geq 1$ and assume that $f:[x_0,\infty)\ra \RR$ is two times
continuously differentiable, $f,f',f''>0$ and that for $x_0\leq x\leq y\leq 2x$
we have $f''(x)/2\leq f''(y)\leq f''(x)$. Let $\lambda,r,T \geq 0$ be
integers, $A\geq 2x_0$ and assume that $a,b$ are integers such that $a\leq b$
and $[a,b)\subseteq [A,2A]$. Then
\begin{multline*}
  \left|
    \left\{
      n\in [a,b):
      s\p{\floor{f(n+\ell)}}-s\p{\floor{f(n+\ell+r)}}
    \right.
  \right.
  \\
  \neq
  \left.
    \left.
      s_\lambda\p{\floor{f(n+\ell)}}-s_\lambda\p{\floor{f(n+\ell+r)}}
      \textrm{ for some }
      \ell\in\{0,\ldots,T-1\}
    \right\}
  \right|
\\
\ll
  (r+T)\p{(b-a)f'(A)2^{-\lambda}+1}
.
\end{multline*}
\end{lemma}
\begin{proof}
The integers $n$ such that $2A-r-T\leq n<b$ contribute an error of at most $r+T$,
therefore it is sufficient to assume that $b\leq 2A-r-T$.
Let $E=(r+T)f'(2A)$. If $E\geq 2^\lambda$, then by~\eqref{eqn:df_quotient} the statement holds trivially.
We assume therefore that $E<2^\lambda$.
Moreover, we may assume that $T\geq 1$, since the statement is vacuous for $T=0$.
Assume that $a\leq n<b$. We first note that if
\begin{equation}\label{eqn:carry_sufficient}
f(n) \in [0,2^\lambda-E)+2^\lambda\ZZ,
\end{equation}
then
\[
s\p{\floor{f(n+\ell)}}-s\p{\floor{f(n+\ell+r)}}
=s_\lambda\p{\floor{f(n+\ell)}}-s_\lambda\p{\floor{f(n+\ell+r)}}
\]
for all $\ell\in\{0,\ldots,T-1\}$.
Indeed, the Mean Value Theorem and the monotonicity of $f'$ imply $f(n+\ell+r)-f(n)\leq E$ and $f(n+\ell)-f(n)\leq E$ for $0\leq \ell<T$.
If~\eqref{eqn:carry_sufficient} holds,
there is therefore an integer $s$ such that $f(n+\ell),f(n+\ell+r)\in[s2^\lambda,(s+1)2^\lambda-1]$ for all $\ell\in\{0,\ldots,T-1\}$, which implies that the binary digits of
$\floor{f(n+\ell)}$ and $\floor{f(n+\ell+r)}$ with indices $\geq \lambda$ are the same.

It remains to count the number of exceptions to~\eqref{eqn:carry_sufficient}.
ø
Since $f$ is increasing to infinity, we may set
$a_s=\min\{n:s2^\lambda\leq f(n)\}$ and
$b_s=\min\{n:(s+1)2^\lambda-E\leq f(n)\}$
for all $s\in\ZZ$.
Clearly we have $f(n)\in[0,2^\lambda-E)+2^\lambda\ZZ$ if $a_s\leq n<b_s$.
It is therefore sufficient to count the number of $n\in[a,b)$ such that
$ b_s\leq n<a_{s+1}$ for some $s$.
By the Mean Value Theorem and~\eqref{eqn:df_quotient} we have
\begin{equation}\label{eqn:interval_length}
\abs{[b_s,a_{s+1})\cap [a,b)}\ll E/f'(A)\ll r+T
\end{equation}
for all $s\in\ZZ$.
Let $s_0$ be minimal such that $a_{s_0}\geq a$.
Moreover, let $s_1$ be maximal such that $a_{s_1}\leq b$.
Then the number of $s$ such that the left hand side of~\eqref{eqn:interval_length} is nonempty is bounded by $s_1-s_0+2$, moreover $a_{s+1}-a_s\gg 2^\lambda/f'(A)$ for $s_0\leq s<s_1$, which follows again from the Mean Value Theorem and~\eqref{eqn:df_quotient}.
Noting that
\[
(s_1-s_0)\min_{s_0\leq s<s_1}(a_{s+1}-a_s)
\leq
\sum_{s_0\leq s<s_1}(a_{s+1}-a_s)
=
a_{s_1}-a_{s_0}
\leq b-a
\]
finishes the proof.

\end{proof}
The inequality of van der Corput is well known. For a proof of this fact see
for example~\cite{MR2010}.
\begin{lemma}\label{lem:vdc}
Let $I$ be a finite interval in $\ZZ$ and
let $a_n$ be a complex number for $n\in I$.
Then
\[
  \abs{\sum_{n\in I}a_n}^2
\leq
  \frac{\abs{I}-1+R}R
  \sum_{n\in I}\abs{a_n}^2
+
  2\frac{\abs{I}-1+R}R\sum_{1\leq r<R}\p{1-\frac rR}
  \mathrm{Re}
  \sum_{\substack{n\in I\\n+r\in I}}a_n\overline{a_{n+r}}
\]
for all integers $R\geq 1$.
\end{lemma}

\noindent
\textbf{Idea of the proof of Proposition~\ref{prp:1}.}
We will prove the result for a more general class of functions $f$, which
yields the generalization announced in the introduction. The function
$f(x)=x^c$ is a special case of such a function. In order to obtain the
nontrivial estimate stated in the proposition, we first apply van der Corput's
inequality to the left hand side. This allows us, using
Lemma~\ref{lem:f_lambda}, to replace the function $s$ by the truncated
sum-of-digits function $s_\lambda$. We then split the summation range $(A,2A]$
into smaller sets, defined by the restrictions $\{f(n)\}\in
\left[m/M,(m+1)/M\right)$ and $k/(LM)\leq f'(n),f'(n+L-1)<(k+1)/(LM)$, where
$L$ and $M$ are chosen later. The idea behind this is that for given $m$ and
$k$, the differences
$\floor{f(n+1)}-\floor{f(n)},\ldots,\floor{f(n+L-1)}-\floor{f(n)}$ should not
depend on the choice of $n$, as long as  $n$ is contained in the set
corresponding to $m$ and $k$. (In fact this will be the case for most $m$,
which is sufficient.) Moreover, as $n$ runs through the set defined by the
above restrictions, the values $\floor{f(n)}$ are uniformly distributed in
residue classes modulo $2^\lambda$. (This step requires the condition
$c<4/3$.) These observations and the $2^\lambda$-periodicity of $s_\lambda$
allow us to remove the function $f$ and to sum over the index set
$\{0,\ldots,2^\lambda-1\}$ instead. We obtain expressions as
in~\eqref{eqn:fourier_coefficients}, for which we have nontrivial estimates by
Lemma~\ref{lem:DMR}.
\section{Proof of the theorem}
Assume that $x_0\geq 1$ and let $f:[x_0,\infty)\ra\RR$ be a two times
continuously differentiable function satisfying the following conditions.
\begin{enumerate}[(a)]
\item \label{hyp:a}
$f,f',f''>0$.
\item \label{hyp:b} For $x_0\leq x\leq y\leq 2x$ we have $f''(x)/2\leq
f''(y)\leq f''(x)$.
\item \label{hyp:c} There exists a $C$ such that $f''(x)\leq
Cx^{-2/3}\log^{-1}x$ for all $x\geq x_0$.
\item \label{hyp:d} For all $k\geq 0$ there exists a $C_k>0$ such that $f''(x)\geq C_k
x^{-1}\log^k(x)$ for all $x\geq x_0$.
\end{enumerate}
We assume for technical reasons that $A\geq A_0$ is an integer,
where $A_0\geq 2x_0$ and $Af''(A)\geq 2$ for all $A\geq A_0$.
By~\eqref{hyp:d} such an $A_0$ exists.
Let $\alpha_0,\ldots,\alpha_{T-1}\in\{0,1\}$. We define
\[
S_0(A,\alpha)=
\sum_{A<n\leq 2A}
\e\p{
\frac 12
\sum_{\ell<T}
\alpha_\ell
s\p{\floor{f(n+\ell)}}
}
.
\]
Our goal is to find a nontrivial bound for this expression. Assume that $R\geq
2$ is an integer. We apply Lemma~\ref{lem:vdc}, which is van der Corput's
inequality:
\begin{multline*}
  \abs{S_0(A,\alpha)}^2
=
  \abs{
    \sum_{A<n\leq 2A}
    \e\p{
      \frac 12\sum_{\ell<T}
      \alpha_\ell
      s\p{\floor{f(n+\ell)} }
    }
  }^2
\\
\leq
  A\frac{A+R}R
+
  2
  \frac{A+R}{R}
  \sum_{1\leq r<R}
  \p{1-\frac{r}{R}}
\\
  \times
  \mathrm{Re}
  \sum_{A<n\leq 2A-r}
  \e\p{
    \frac 12
    \sum_{\ell<T}
    \alpha_\ell
    \p{
        s\p{\floor{f(n+\ell)}}
      -
        s\p{\floor{f(n+\ell+r)}}
    }
  }.
\end{multline*}
We replace $s$ by the truncated sum-of-digits function $s_\lambda$ by means of
Lemma~\ref{lem:f_lambda}. Moreover, we replace the summation limit $2A-r$ by
$2A$ and obtain
\begin{multline}\label{eqn:vdc_applied}
  \abs{S_0(A,\alpha)}^2
\ll
  O\p{\frac{A^2}R+AL+A^2\frac{Lf'(A)}{2^\lambda}}
\\
+
  \frac AR
  \sum_{1\leq r<R}
  \p{1-\frac{r}{R}}
  \mathrm{Re}
  \sum_{A<n\leq 2A}
  \e\p{
    \frac 12
    \sum_{\ell<L}
    \beta_{\ell,r}
    s_\lambda\p{\floor{f(n+\ell)}}
  }
,
\end{multline}
where $L=T+R-1$ and $\beta_{\ell,r}$ is chosen properly: if we set
$\alpha_\ell=0$ for $\ell\not\in [0,T)$, we may choose
$\beta_{\ell,r}=\alpha_\ell-\alpha_{\ell-r}$ for $0\leq \ell<L$. In
particular, we have $\beta_{0,r}=1$ and $\sum_\ell \beta_{\ell,r}\in 2\ZZ$ for
all $r\geq 1$. We are therefore concerned with expressions of the form
\begin{equation}\label{eqn:S1_definition}
  S_1(A,\beta)
=
  \sum_{A<n\leq 2A}
  \e\p{
    \frac 12
    \sum_{\ell<L}
    \beta_\ell
    s_\lambda\p{\floor{f(n+\ell)}}
  }
,
\end{equation}
where $\beta_0=1$ and $\sum_\ell \beta_\ell\in 2\ZZ$, which we want to estimate
nontrivially. In order to do so, we dissect the interval $(A,2A]$ into smaller
pieces as follows. Let $M\geq 1$ be an integer to be chosen later and set
\begin{equation}\label{eqn:J_definition}
  J(k,L,M)
=
  \left\{
    n:\frac k{LM}\leq f'(n),f'(n+L-1)<\frac{k+1}{LM}
  \right\}
.
\end{equation}
Moreover, we set $d_0=d_0(A)=\floor{f'(A)}+1$ and $d_1=d_1(A)=\floor{f'(2A)}$.
We have $f'(A)<d_0\leq d_1\leq f'(2A)$, the second inequality being justified by the assumption $Af''(A)\geq 2$: using the Mean Value Theorem, hypothesis~\eqref{hyp:b} and this assumption, we get
$d_1-d_0\geq f'(2A)-f'(A)-1\geq Af''(A)/2-1\geq 0.$
Similarly, we obtain
\begin{equation}\label{eqn:slope}
d_1-d_0\leq A\,f''(A).
\end{equation}
By monotonicity of $f'$ the sets $J(k,L,M)$ are intervals in $\ZZ$, we may
therefore choose integers $a$ and $b$ such that $J(k,L,M)=[a,b)$. Assume that
$d_0LM\leq k<d_1LM$. Then $J(k,L,M)\subseteq (A,2A]$. Since $R\geq 2$ and
$T\geq 1$, we have $L-1\geq 1$ and therefore
$1/(LM)\geq f'(b)-f'(a)=(b-a)f''(\xi)\geq (b-a)f''(A)/2$ by (b), which implies
\begin{equation}\label{eqn:J_size}
\abs{J(k,L,M)}
\leq \frac 2{f''(A)LM}
.
\end{equation}
Clearly we have
\begin{equation}\label{eqn:disjoint_intervals}
\sum_{d_0LM\leq k<d_1LM}\abs{J(k,L,M)}\leq A,
\end{equation}
since the sets $J(k,L,M)$ are pairwisely disjoint and contained in $(A,2A]$.
Moreover, they almost cover the interval $(A,2A]$: we have
\[
\sum_{d_0LM\leq k<d_1LM}\abs{J(k,L,M)}\geq A-O\p{\frac 1{f''(A)}+A\,f''(A)L^2M}
,
\]
where the first error term takes care of the integers $n$ such that
$f'(A)<f'(n)\leq d_0$ or $d_1<f'(n)\leq f'(2A)$,
which is again an application of the Mean Value Theorem,
and the second error term covers the integers $n$ that are excluded by the
second condition in~\eqref{eqn:J_definition} --- the length of the summation
over $k$ is estimated via~\eqref{eqn:slope}, and for each $k$ we take out at
most $L$ integers. We obtain
\begin{multline}\label{eqn:S1}
  S_1(A,\beta)
=
  \sum_{d_0LM\leq k<d_1LM}
  \sum_{m<M}
  \sum_{\substack{n\in J(k,L,M)\\\frac mM\leq \{f(n)\}<\frac{m+1}M}}
  \e\p{
    \frac 12
    \sum_{\ell < L}
    \beta_\ell
    s_\lambda\p{\floor{f(n+\ell)}}
  }
\\
+
  O\p{
      \frac 1{f''(A)}
    +
      A\,f''(A)L^2M
  }
.
\end{multline}
We define
\begin{equation}\label{eqn:dfn_S2}
  S_2(k,m,L,M,\beta)
=
  \sum_{\substack{n\in J(k,L,M)\\\frac mM\leq \{f(n)\}<\frac{m+1}M}}
  \e\p{
    \frac 12
    \sum_{\ell< L}
    \beta_\ell
    s_\lambda\p{\floor{f(n+\ell)}}
  }
,
\end{equation}
for which we have to find an estimate. We define a set of ``good'' $m$ by
\[
  G(k,L,M)
=
  \left\{
    m<M:
    \left[
      \frac mM+\ell \frac{k}{LM},\frac{m+1}M+\ell\frac{k+1}{LM}
    \right)
    \cap\ZZ
    =
    \emptyset
    \textrm{ for }0\leq \ell<L
  \right\}
.
\]
We claim that
\begin{equation}\label{eqn:excluded_m}
\abs{G(k,L,M)}\geq M-O(L)
.
\end{equation}
Indeed, the intervals in the definition of $G(k,L,M)$ have length $\leq 2/M$,
therefore for each $\ell$ we have to exclude only $O(1)$ integers $m<M$.

For $m\not\in G(k,L,M)$ we estimate $S_2$ trivially. To this end, we use
Lemma~\ref{lem:f_discrepancy} with $K=1$ in order to count the number of
summands. We obtain with the help of (\ref{eqn:J_size})
\begin{multline}\label{eqn:slice_estimate}
\abs{
  \left\{
    n\in J(k,L,M):\frac mM\leq \{f(n)\}<\frac{m+1}M
  \right\}
}
\\
\ll
  \frac{\abs{J(k,L,M)}}{M}
+
  \p{f''(A)LM}^{-1/2}
.
\end{multline}
Combining~\eqref{eqn:slope}, (\ref{eqn:disjoint_intervals}), (\ref{eqn:S1}),
(\ref{eqn:dfn_S2}), (\ref{eqn:excluded_m}) and~\eqref{eqn:slice_estimate} gives
\begin{multline}\label{eqn:S1_continued}
  S_1(A,\beta)
=
  \sum_{d_0LM\leq k<d_1LM}
  \sum_{m\in G(k,L,M)}
  S_2(k,m,L,M,\beta)
\\
+
  O\p{
      A\frac LM
    +
      A f''(A)^{1/2} L^{3/2}M^{1/2}
    +
      \frac 1{f''(A)}
    +
      A\,f''(A)L^2M
  }
.
\end{multline}
Next we want to remove the function $f$ occurring in the sum $S_2$. In order
to do so, we will use the argument that the values of $\floor{f(n)}$ are
approximately uniformly distributed in residue classes modulo $2^\lambda$,
where $n$ satisfies the restrictions under the sum in the definition of $S_2$.
If $n\in J(k,L,M)$ is such that $m/M\leq \{f(n)\}<(m+1)/M$, then
\[
  \floor{f(n)}+\frac mM+\ell \frac k{LM}
\leq
  f(n+\ell)
<
  \floor{f(n)}+\frac{m+1}M+\ell \frac{k+1}{LM}
\]
for all $\ell<L$, which follows easily from the definition of $J(k,L,M)$ and
the Mean Value Theorem. Consequently, if $m\in G(k,L,M)$, then
\[
  \floor{f(n+\ell)}
=
  \floor{f(n)}
+
  \floor{\frac mM+\ell \frac k{LM}}
=
  \floor{f(n)}
+
  \ell\floor{\frac{k}{LM}}
+
  i_\ell^{k,m},
\]
where the correction terms $i^{k,m}_\ell$ satisfy $i^{k,m}_0=0$ and
$i_{\ell+1}^{k,m}-i_\ell^{k,m}\in \{0,1\}$. For any $m\in G(k,L,M)$ we obtain
therefore
\begin{multline}\label{eqn:S2_decomposition}
  S_2(k,m,L,M,\beta)
\\
=
  \sum_{\substack{n\in J(k,L,M)\\
  \frac mM\leq \{f(n)\}<\frac{m+1}M}}
  \e\p{
    \frac 12
    \sum_{\ell<L}
    \beta_\ell
    s_\lambda\p{
      \floor{f(n)}+\ell\floor{\frac{k}{LM}}+i^{k,m}_\ell
    }
  }
\\
=
  \sum_{s<2^\lambda}
  \sum_{\substack{n\in J(k,L,M)\\
  \frac mM\leq \{f(n)\}<\frac{m+1}M\\
  \floor{f(n)}\equiv s\bmod 2^\lambda}}
  \e\p{\frac 12
  \sum_{\ell<L}
  \beta_\ell
  s_\lambda\p{
    s+\ell\floor{\frac{k}{LM}}+i^{k,m}_\ell
  }
}
.
\end{multline}
The summand does not depend any more on $n$, so that we only have to count the
number of times the three conditions under the second summation sign are
satisfied. For this purpose we use Lemma~\ref{lem:f_discrepancy} again, this
time taking $K=2^\lambda$. We obtain for $m\in G(k,L,M)$
\begin{multline}\label{eqn:S2_beatty}
S_2(k,m,L,M,\beta)
=
O\p{
    \frac{\abs{J(k,L,M)}}{LM}
  +
    2^\lambda
    \p{f''(A)LM}^{-1/2}
}
\\
+
  \frac{\abs{J(k,L,M)}}M
  \frac 1{2^\lambda}
  \sum_{s<2^\lambda}
  \e\p{
    \frac 12
    \sum_{\ell<L}
    \beta_\ell
    s_\lambda\p{
      s+\ell\floor{\frac{k}{LM}}+i^{k,m}_\ell
    }
  }
.
\end{multline}
This process is valid for each $k$ and $m\in G(k,L,M)$. In order to sum this
expression over $k$ and $m$, which is needed in order to return to the sum
$S_1$, we want to ``forget'' the upper indices in $i^{k,m}_\ell$. We consider
therefore the sum
\begin{multline}\label{eqn:S3}
  S_3(A,L,M,i,\beta,\lambda)
\\
=
  \sum_{d_0LM\leq k<d_1LM}
  \abs{
    \frac 1{2^\lambda}
    \sum_{s<2^\lambda}
    \e\p{
      \frac 12
      \sum_{\ell<L}
      \beta_\ell
      s_\lambda\p{
      s+\ell\floor{\frac{k}{LM}}+i_\ell
      }
    }
  }
.
\end{multline}
(Note that $d_0$ and $d_1$ are functions of $A$.) We have
\[
  S_3(A,L,M,i,\beta,\lambda)
=
  \sum_{d_0LM\leq k<d_1LM}
  \abs{
    G_\lambda^{i,\beta}\p{0,\floor{\frac{k}{LM} } }
  }
\leq
  LM
  \sum_{0\leq d<2^{\lambda'}}
  \abs{G_\lambda^{i,\beta}(0,d)}
,
\]
where $\lambda'$ is chosen in such a way that
$2^{\lambda'-1}<d_1\leq 2^{\lambda'}$. We note that $f'(A)\asymp 2^{\lambda'}$
by~\eqref{eqn:df_quotient}, therefore we obtain by Lemma~\ref{lem:DMR} and the
Cauchy--Schwarz inequality
\begin{equation}\label{eqn:S3_estimate}
  S_3(A,L,M,i,\beta,\lambda)
\leq
  c_1
  M
  f'(A)^{1-\eta}
\end{equation}
for all $\lambda$ satisfying the restriction
$2^{\lambda/2}\leq \floor{f'(2A)}\leq 2^\lambda$, where $c_1$ and $\eta>0$
depend only on $L$.

We take the sum of~\eqref{eqn:S2_beatty} over $k$,
comprising $(d_1-d_0)LM\leq A\,f''(A)LM$ summands,
and over $m$, comprising $\abs{G(k,L,M)}\leq M$ summands,
and use the estimates $\abs{J(k,L,M)}\leq 2/(f''(A)LM)$
and~\eqref{eqn:S3_estimate}, which yields
\begin{multline*}
  \sum_{d_0LM\leq k<d_1LM}
  \sum_{m\in G(k,L,M)}
  S_2(k,m,L,M,\beta)
\ll
  \frac 1{f''(A)LM}
\\
  \times
  \sum_{d_0LM\leq k<d_1LM}
  \sup_{\substack{i_0,\ldots,i_{L-1}\\i_0=0\\i_{\ell+1}-i_\ell\in\{0,1\}} }
  \abs{
    \frac 1{2^\lambda}
    \sum_{s<2^\lambda}
    \e\p{
      \frac 12
      \sum_{\ell<L}
      \beta_\ell
      s_\lambda\p{
      s+\ell\floor{\frac{k}{LM}}+i_\ell
      }
    }
  }
\\
+
  \frac AL+Af''(A)^{1/2}2^\lambda L^{1/2}M^{3/2}
  \ll
  c_2 \frac{f'(A)^{1-\eta}}{f''(A)}
+
  \frac AL
+
  Af''(A)^{1/2}2^\lambda L^{1/2}M^{3/2}
,
\end{multline*}
where $c_2$ and $\eta$ depend only on $L$.
Combining this with~\eqref{eqn:S1_continued}, we get
\begin{equation*}
  S_1(A,\beta)
\ll
  c_2
  \frac{f'(A)^{1-\eta}}{f''(A)}
+
  \frac AL
+
  A\frac LM
+
  Af''(A)^{1/2}
  2^\lambda L^{3/2}M^{3/2}
+
  \frac 1{f''(A)}
+
  A\,f''(A)L^2M
.
\end{equation*}
Finally, we have to take $R$ into account and treat the sum $S_0$.
Using~\eqref{eqn:vdc_applied}, we obtain uniformly for $A,R,M,\lambda$ such
that $A\geq A_0$, $R\geq 2$, $M\geq 2$ and
$2^{\lambda/2}\leq \floor{f'(2A)}\leq 2^\lambda$ (this latter condition is
needed for~\eqref{eqn:S3_estimate})
\begin{multline}\label{eqn:cumulative_error}
  \abs{S_0(A,\alpha)}^2
\ll
  \frac {A^2}R+AL+A^2\frac{Lf'(A)}{2^\lambda}
+
  A \sup_{\beta} \abs{S_1(A,\beta)}
\\
\ll
  \frac {A^2}R
+
  AL
+
  A^2\frac{Lf'(A)}{2^\lambda}
+
  c_2(L)
  A
  \frac{f'(A)^{1-\eta(L)}}{f''(A)}
+
  A^2\frac LM
\\
+
  A^2f''(A)^{1/2}
  2^\lambda L^{3/2}M^{3/2}
+
  \frac A{f''(A)}
+
  A^2f''(A)L^2M
,
\end{multline}
where $L=R+T-1$. We have to choose $R,M$ and $\lambda$. Assume that
$0<\varepsilon<1$ and let $R\geq 2$ be the minimal integer such that $1/R<\varepsilon$.
Moreover, let $M=RL$. Then obviously $L/M<\varepsilon$.
By~\eqref{hyp:d} we have $Af''(A)\ra\infty$,
which implies that $A/f''(A)=o(A^2)$, moreover
we have $f'(A)\ra\infty$, as we noted in the proof of Lemma~\ref{lem:f_lambda}.
This fact, using~\eqref{eqn:df_d2f_2}, implies $f'(A)\leq 3Af''(A)\log A$ for
$A$ large enough. Using also~\eqref{hyp:d}, we get for all $k$
\begin{multline*}
  \frac{f'(A)^{1-\eta}}{f''(A)}
\ll
  \frac{A^{1-\eta}f''(A)^{1-\eta}(\log A)^{1-\eta}}{f''(A)}
\ll
  A^{1-\eta}f''(A)^{-\eta}(\log A)^{1-\eta}
\\
\ll
  (C_k)^{-\eta}A \log^{1-\eta (k+1)}A
.
\end{multline*}
For $k$ large enough this is $o(A)$.
For given $A$ choose $\lambda=\lambda(A)$ minimal such that
$Lf'(2A)/2^\lambda\leq \varepsilon$. Then for large $A$ the condition
$2^{\lambda/2}\leq \floor{f'(2A)}\leq 2^\lambda$, which we need
for~\eqref{eqn:cumulative_error}, is satisfied. Using the definition of
$\lambda$, the estimate~\eqref{eqn:df_d2f_2} and hypothesis~\eqref{hyp:c} (in
this order), we get
\[
  f''(A)^{1/2}2^\lambda
\ll
  f''(A)^{1/2}\varepsilon^{-1} Lf'(A)
\ll
  \varepsilon^{-1} L f''(A)^{3/2}A\log A
\ll
  C^{3/2}
  \varepsilon^{-1} L \log^{-1/2}A
\]
for $A$ large enough, which also tends to $0$ as $A\ra\infty$. Finally, we
note that $f''(A)\ra 0$ by~\eqref{hyp:c}. These observations are sufficient to
prove that each of the error terms in~\eqref{eqn:cumulative_error} is bounded
by $\varepsilon A^2$ if $A$ is large enough. In particular, this proves
Proposition~\ref{prp:1} and therefore the theorem.

\begin{remark}
As we announced in the introduction, we proved the normality of more general
subsequences of $\mathbf t$ than those indexed by $\floor{n^c}$. For a given
function $f$, we only have to find an $x_0$ such that the
hypotheses~\eqref{hyp:a} to~\eqref{hyp:d} are satisfied. Not only can we
handle obvious variations such as $n\mapsto\floor{n^c(\log n)^\eta}$ or
$n\mapsto \floor{n^{c_1}+n^{c_2}}$, we can also take functions such as
$n\mapsto\floor{n\exp{\log^{1-\varepsilon} n}}$, where $0<\varepsilon<1$, or
$n\mapsto\floor{n \exp((\log\log n)^{1+\varepsilon})}$, growing more slowly
than $n^{1+\varepsilon}$ for all $\varepsilon>0$. On the other hand, it has
been shown by Deshouillers, Drmota and Morgenbesser~\cite{DDM2012} that the
asymptotic densities of $0$ and $1$ in the subsequence of $\mathbf t$ indexed
by $\floor{n\log n}$ do not exist, in particular we do not obtain a normal
sequence in this way. It would therefore be interesting to locate more
precisely the rate of growth where the ``phase transition'' takes place. For
example, the question whether the sequence
$n\mapsto \mathbf t(\lfloor n \log^k n\rfloor)$ is normal for some $k$ remains
open at the moment.
\end{remark}

\section*{Funding}
This work was supported by the Austrian Science Fund (FWF) [project F5502-N26,
which is a part of the Special Research Program
``Quasi Monte Carlo Methods: Theory and Applications''].

\bibliographystyle{siam}
\bibliography{normality}

\end{document}